\documentclass[12pt]{article}

\usepackage{amssymb,amsmath,amsbsy, amsthm}
\usepackage{amscd}
\usepackage{pb-diagram}
\usepackage{pstricks,pst-node}
\usepackage{amsfonts}
\usepackage{latexsym}
\usepackage{pstricks,pst-node}
\usepackage{tikz}
\usepackage{fullpage}
\usepackage{xcolor}

\usepackage{hyperref}
\hypersetup{
     colorlinks   = true,
     citecolor    = blue,
    linkcolor = blue
}

\swapnumbers\newtheorem{theorem}{Theorem}[section]
\newtheorem{lemma}[theorem]{Lemma}
\newtheorem{corollary}[theorem]{Corollary}

\newtheorem{affirmation}[theorem]{Affirmation}
\newtheorem{proposition}[theorem]{Proposition}
\newtheorem{example}[theorem]{Example}
\newtheorem{problem}[theorem]{Problem}
\newtheorem{question}[theorem]{Question}
\theoremstyle{definition}

\newtheorem{remark}[theorem]{Remark}

\numberwithin{equation}{section}



\newcommand{\Cl}{\mathrm{Cl}}

\def\NN{{\hbox{{\rm I$\!$\rm N}}}}

\def\tto{\ \hbox{$\to\!\!\!\!\!\to$}\ }
\newcommand{\N}{\mathbb{N}}

\newcommand{\T}{\mathcal{T}}

\newcommand{\C}{\mathcal{C}}

\newcommand{\G}{\mathcal{G}}
\newcommand{\HH}{\mathcal{H}}

\newcommand{\OO}{\mathcal{O}}

\newcommand{\V}{\mathcal{V}}

\begin{document}

\date{}
\title{Orbit sets, transitivity, and sensitivity with upper semicontinuous maps}

\author{Jeison Amorocho, Javier Camargo and Sergio Mac\'{\i}as
}

\maketitle

\begin{abstract}
Given a compact metric space $X$ and an upper semicontinuous function $F\colon X \to 2^X$, we explore the dynamic system $(X,F)$. In this study, we introduce new concepts, demonstrate various results, and provide numerous examples. In particular, we define the orbit set $\mathcal{O}_F(p)$ and prove that it is compact. We also establish conditions for connectedness of the orbit sets and pose several questions related to the system.

 We also investigate transitivity and its relation to the density of orbits. In addition, we present strong and weak notions of sensitivity and examine the relationships between these concepts.  
\end{abstract}

\section{Introduction}\label{Intro}

A discrete dynamic system is defined as a pair $(X, f)$, where $X$ is a compact metric space, and $f\colon X\to X$ is a continuous function. Discrete dynamical systems can be extended to include set-valued functions $F\colon X \to 2^X$, where $2^X$ represents the collection of all closed nonempty subsets of $X$, and $F$ is an upper semicontinuous function. Thus, a multivalued discrete dynamical system is a pair $(X,F)$, where $X$ is a compact metric space, and $F\colon X\to 2^X$ is an upper semicontinuous function. This generalisation has been explored in various studies, including \cite{CMM}, \cite{CM}, \cite{CGMO}, and \cite{RT}. If the upper semicontinuous function $F$ satisfies $F(x) = \{ y_x \}$ for each $x \in X$, then the function $f \colon X \to X$ defined by $f(x) = y_x$ is continuous. Thus, upper semicontinuity can be seen as a natural extension of continuity in the context of multivalued functions. Moreover, all the dynamical concepts introduced in this paper for multivalued functions coincide with the classical definitions when the multivalued function is indeed a single-valued function.

The paper is divided into six sections. After this Introduction, we have a section~\ref{Def} of definitions,
where we include the material needed for the paper.
We investigate the dynamical system $(X, F)$, where $X$ is a compact metric space and $F\colon X \to 2^X$ is an upper semicontinuous function. We introduce new concepts, prove various results, and provide numerous examples throughout our study. 
Specifically, for a point $p \in X$, we define the orbit set $\mathcal{O}_F(p)$ as the collection of all sequences $(x_n)_{n\in\mathbb{N}} \in X^{\mathbb{N}}$ such that $x_1 = p$ and $x_{n+1} \in F(x_n)$ for every $n \in \mathbb{N}$.  In Section~\ref{Orbits}, we explore certain topological properties of the orbit set, present examples, and pose questions related to the family of orbits originating from a given point. In particular, in Theorem~\ref{theo87g5}, we establish that the orbit set $\mathcal{O}_F(p)$ is compact. Moreover, we prove that if $F(x)$ is a continuum for every $x \in X$, then $\mathcal{O}_F(p)$ is also a continuum (see Theorem \ref{theom}).

Section~\ref{MoreOrbits} is dedicated to the presentation of notable examples of orbit sets and to the investigation of the relationship between orbits and inverse limits. In Proposition \ref{ex0}, an upper semicontinuous function $F\colon [0,1] \to 2^{[0,1]}$ is defined such that the orbit set $\mathcal{O}_F(p)$ constitutes an infinite-dimensional continuum for some $p\in [0,1]$, thereby illustrating its complexity.

Furthermore, in Examples \ref{exm8uh6} and \ref{exoif6hdt7}, two distinctive dendrites are provided, which can be interpreted as orbit sets when $F$ is defined in the closed interval $[0, 1]$. These examples not only serve to exemplify the diverse structures that orbit sets can assume, but also encourage further investigation into their characteristics.

In Section~\ref{DensOrbTrans}, we define the concepts of transitivity in the context of an upper semicontinuous function and establish conditions for the existence of either a dense orbit or a weak dense orbit. This section focusses on comparing these three concepts. Additionally, towards the end of the section, we introduce the terms ``dense minimal'' and ``weak dense minimal.'' Theorem \ref{theodminimal} shows that these two notions are equivalent.

In Section~\ref{Sensitive}, we explore various ways to introduce the concept of sensitivity in the context of multivalued functions. We propose several definitions related to sensitivity: strongly sensitive, sensitive, weakly sensitive, and Li-Yorke sensitive. According to our definitions, Proposition \ref{prop8uy61} establishes that ``strong sensitivity" implies ``sensitivity," and ``sensitivity" implies ``weak sensitivity." Additionally, Proposition \ref{propbshdgst} shows that ``Li-Yorke sensitivity" implies ``sensitivity." We also provide examples that demonstrate that the converse implications do not hold in general.  

\section{Definitions}\label{Def}

A \textit{compactum} is a compact metric space.
A {\it continuum} is a connected compactum. A
{\it subcontinuum} is a continuum contained in a metric space. An \textit{arc} is any continuum homeomorphic to $[0,1]$. A \textit{simple closed curve}
is a continuum homeomorphic to $S^1=\{z\in \mathbb C : |z|=1\}$. A \textit{dendrite} is a locally connected continuum that does not contain simple closed curves. 
A \textit{map} is a continuous function. Given a continuum $X$, we say that it is an \textit{arc-like continuum} provided that for each $\varepsilon>0$, there exists an onto map $f\colon X\to [0,1]$ such that $\mathrm{diam}(f^{-1}(t))<\varepsilon$, for each $t\in [0,1]$. 
A continuum $X$ is {\it decomposable} if there exist two proper subcontinua $A$ and $B$ of $X$ such that $X=A\cup B$. If $X$ is not decomposable, we say that $X$ is {\it indecomposable}.
A continuum $X$ is said to be {\it aposyndetic} provided that for each $p$ of $X$ and for any $q\in X\setminus\{p\}$, there exists a subcontinuum $W$ of $X$ such that $p\in \mathrm{Int}(W)$ and $q\notin W$. Also, $X$ is {\it homogeneous} if for any $x,y\in X$, there exists a homeomorphism $h\colon X\to X$ such that $h(x)=y$.
We denote $2^X$ as the family of all compact nonempty subsets of $X$. 

Given a compactum $X$ and a map $f\colon X\to X$, then:
\begin{itemize}
    \item For each $p\in X$, the \textit{orbit of }$p$ is the sequence $\mathcal{O}_f(p)=\{p,f(p),f^{2}(p),\ldots\}$, where $f^{n}=f\circ f\circ \cdots \circ f$, $n$ times.
    \item We say that $f$ is \textit{transitive} if for each pair of open sets $U$ and $V$, there exists $z\in U$ such that $f^m(z)\in V$, for some $m\in\mathbb N$.
    \item A point $p\in X$ has a \textit{ dense orbit} if the set $\mathcal{O}_f(p)$ is a dense subset of $X$.
    \item $f$ is \textit{dense minimal} provided that $x$ has dense orbit, for each $x\in X$.
\end{itemize}

A proof of the following result can be found in \cite[Proposition~39, p.155]{Block}.

\begin{proposition}\label{propB}
    Let $X$ be a compactum with no isolated points and let $f\colon X\to X$ be a map. Then $f$ is transitive if and only if there exists $p\in X$ such that $p$ has dense orbit.
\end{proposition}

Given a compactum $X$, we define Professor Jones' 
{\it Set Function $\T$} as follows: if $A$ is a subset of $X$, then
\begin{multline}\label{defT} 
\T(A)=X\setminus\{x\in X : \hbox{\rm there exists a subcontinuum}\\
\hbox{\rm $W$ of $X$ such that $x\in Int(W)\subset W\subset$ $X\setminus A$}\}.
\end{multline} 
Note that if $W$ is a subcontinuum of $X$, then $\T(W)$ is a subcontinuum
of $X$ \cite[Theorem~2.1.27]{MS}. Hence, $\T(\{x\})$ is a subcontinuum of $X$ for each $x\in X$. The set function $\T$ is {\it idempotent on closed sets} if
$\T^2(A)=\T(A)$, for all closed $A$ of $X$. For a comprehensive study of this topic, see \cite{MS}.

\section{Orbit Sets}\label{Orbits}

Given $X$ a compactum and a function $F\colon X\to \mathcal{P}(X)$, we define the \textit{graph} of $F$ by $$\mathrm{Gr}(F)=\{(x,y)\in X\times X : y\in F(x)\}.$$
For each $z\in X$, let
\begin{equation}
    \mathcal{O}_F(z)=\left\{(x_n)_{n\in\mathbb N}\in X^{\mathbb N} : x_1=z \text{ and }x_{n+1}\in F(x_n), \text{ for each }n\in\mathbb N\right\}.
\end{equation}
A sequence $(x_i)_{i=1}^{\infty}$ is called an \textit{orbit} of $z$ provided that $(x_i)_{i=1}^{\infty}\in \mathcal{O}_F(z)$. We say that $\mathcal{O}_F(z)$ is the \textit{orbit set of $z$}. In this section, we study some topological properties of $\mathcal{O}_F(z)$.

\medskip

Given $F\colon X\to \mathcal{P}(X)$ and a subset $A$ of $X$, we define $F(A)=\bigcup\{F(x) : x\in A\}$. If $G\colon X\to \mathcal{P}(X)$ is a function, then we define $G\circ F\colon X\to \mathcal{P}(X)$ by $(G\circ F)(x)=G(F(x))$, for each $x\in X$.
Hence, inductively for every $n\in\mathbb N$, we define $F^n\colon X\to \mathcal{P}(X)$ by:
\begin{equation}\label{eq00}
  F^{n+1}=F\circ F^n.  
\end{equation}

\begin{lemma}\label{lem0}
Let $X$ be a compactum and let $F\colon X\to \mathcal{P}(X)$ be a 
function. If $z\in X$, then 
$$\pi_k(\mathcal{O}_F(z))=F^{k-1}(z),$$ 
for each $k\geq 2$.
\end{lemma}

\begin{proof}
Note that $\pi_2(\mathcal{O}_F(z))=F(z)$. We prove the lemma by 
induction. Suppose that $\pi_k(\mathcal{O}_F(z))=F^{k-1}(z)$, for some 
$k\geq 3$. We show that $\pi_{k+1}(\mathcal{O}_F(z))=F^{k}(z)$. Let $x\in \pi_{k+1}(\mathcal{O}_F(z))$. 
Then there exists $(z, z_2, z_3,\ldots)\in\mathcal{O}_F(z)$ such that $z_{k+1}=x$. 
Hence, $x\in F(z_k)$, where $z_k\in \pi_k(\mathcal{O}_F(z))$. Since 
$\pi_k(\mathcal{O}_F(z))=F^{k-1}(z)$, $z_k\in F^{k-1}(z)$. Thus, 
$x\in F(z_{k})$ and $F(z_{k})\subseteq F(F^{k-1}(z))$. 
Therefore, $x\in F^{k}(z)$ (see (\ref{eq00})). Hence, $\pi_{k+1}(\mathcal{O}_F(z))\subseteq F^k(z)$. 
Conversely, let $x\in F^k(z)$. 
By (\ref{eq00}), there exists $w_k\in F^{k-1}(z)$ such that $x\in F(w_k)$. 
Also, we know that there exists $w_{k-1}\in F^{k-2}(z)$ where $w_k\in F(w_{k-1})$. 
In this way, we find $w_2,\ldots , w_{k}$ such that $w_2\in F(z)$, $x\in F(w_k)$, 
and $w_j\in F(w_{j-1})$, for each $j\in\{3,\ldots,k\}$. 
Now, let $w_{k+2}\in F(x)$, and 
let $w_{k+2},w_{k+3},\ldots$ be such that $w_j\in F(w_{j-1})$ for all 
$j\geq k+3$. 
Thus, we define a sequence $(z,w_2,\ldots)\in \mathcal{O}_F(z)$ where 
$\pi_{k+1}((z,w_2,\ldots))=x$, and $x\in \pi_{k+1}(\mathcal{O}_F(z))$. Thus,  $F^{k}(z)\subseteq \pi_{k+1}(\mathcal{O}_F(z))$.
Therefore, $\pi_{k+1}(\mathcal{O}_F(z))=F^{k}(z)$.
\end{proof}

A function $F\colon X\to \mathcal{P}(X)$ is \textit{upper semicontinuous} provided that for each $x\in X$ and every open subset $V$ of $X$ such that $F(x)\subseteq V$, there exists an open subset $U$ of $X$, where $x\in U$ and $F(y)\subseteq V$, for all $y\in U$. 
We say that $F$ is \textit{lower semicontinuous} if for each $x\in X$ and every open subset $V$ such that $F(x)\cap V\neq\emptyset$, there exists an open subset $U$ such that $x\in U$ and $F(y)\cap V\neq\emptyset$, for each $y\in U$. 

Given a compactum $X$, we denote $2^X$ the family of all nonempty closed subsets of $X$. If $M$ is a subset of $X$, we denote $\langle M\rangle=\{ A\in 2^X : A\subseteq M\}$ and $\langle X, M\rangle =\{A\in 2^X : A\cap M\neq\emptyset\}$. Let $(A_n)_{n=1}^{\infty}$ be a sequence of $\mathcal{P}(X)$. We define:
\begin{multline*}
    \bullet \ \liminf A_n=\{x\in X : \text{ for each open }U\text{ such that }\\ x\in U,\ U\cap A_i\neq\emptyset \text{ for all but finitely many indices }i\};
\end{multline*}
\begin{multline*}
    \bullet \ \limsup A_n=\{x\in X : \text{ for every open }U\text{ such that }\\ x\in U,\ U\cap A_i\neq\emptyset \text{ for infinitely many indices }i\}.
\end{multline*}

\begin{proposition}\label{prophfg6tg}
    Let $X$ be a compactum and let $F\colon X\to 2^X$ be a function. Then the following are equivalent:
    \begin{enumerate}
        \item $F$ is lower semicontinuous;
        \item $F^{-1}(\langle X,V\rangle)$ is open, for each open set $V$ of $X$; and
        \item $F(x)\subseteq \liminf F(x_n)$ for each sequence $(x_n)_{n=1}^{\infty}$ such that $\lim_{n\to\infty}x_n=x$.
    \end{enumerate}
\end{proposition}

\begin{proof}
    We have, by definition, that \textit{1} and \textit{2} are equivalent. Observe that \cite[Theorem 2, p. 62]{Kur2} gives the equivalence between the statements \textit{2} and \textit{3}.
\end{proof}

\begin{proposition}\label{prop7uh}
    Let $X$ be a compactum and let $F\colon X\to 2^X$ be a function. Then the following are equivalent:
    \begin{enumerate}
        \item $F$ is upper semicontinuous;
        \item $F^{-1}(\langle V\rangle)$ is open, for each open set $V$ of $X$; and
        \item $\limsup F(x_n)\subseteq F(x)$ for each sequence $(x_n)_{n=1}^{\infty}$ such that $\lim_{n\to\infty}x_n=x$;
        \item $\mathrm{Gr}(F)$ is closed.
    \end{enumerate}
\end{proposition}

\begin{proof}
    Note that \textit{1} and \textit{2} are equivalent by definition, and that \textit{2} and \textit{3} are equivalent by \cite[Theorem 1, p. 61]{Kur2}. Also, according to \cite[Theorem 4, p. 58]{Kur2}, we have that \textit{2} is equivalent to \textit{4}.
\end{proof}

A function $F\colon X\to 2^X$ is \textit{continuous} provided that $F$ is both lower and upper semicontinuous. 

\begin{remark}
    Given a compactum $X$, the family $\mathcal{B}=\{\langle U\rangle : U \text{ is open of }X\}\cup\{\langle X,V\rangle : V \text{ is open of }X\}$ is a base of Vietoris topology on $2^X$\cite[Definition (0.12)]{N1}. Thus, a function $F\colon X\to 2^X$ is continuous if and only if $F$ is continuous where $2^X$ is endowed with Vietoris topology.
\end{remark}

\begin{remark}
    If $F$ is upper semicontinuous and $A$ is a closed subset of $X$, by the Proposition~\ref{prop7uh} part \textit{3}, we have $F(A)\in 2^X$. Hence, if $F\colon X\to 2^X$ is upper semicontinuous, then $F^n\colon X\to 2^X$ is well defined for each $n\in\mathbb N$ (see (\ref{eq00})).
\end{remark}

\begin{remark}\label{defTT}
    Let $X$ be a compactum and let $T\colon X\to 2^X$ be defined by $T(x)=\T(\{x\})$ for each $x\in X$, where $\T$ is the Jones' function defined in (\ref{defT}). By \cite[Theorem 3.3.1]{MT}, $T$ is an upper semicontinuous function.
\end{remark}

By \cite[Theorem~1.4]{I} and mathematical induction, we obtain following result.

\begin{theorem}\label{theo9iuh67}
    Let $X$ be a compactum and let $F\colon X\to 2^X$ be an upper semicontinuous function. Then $F^{n}\colon X\to 2^X$ is upper semicontinuous, for each $n\in\mathbb N$.
\end{theorem}

In Theorem~\ref{theo9yr34rs}, we show a similar result to Theorem~\ref{theo9iuh67} for lower semicontinuous functions.

\begin{theorem}\label{theo8ucpa3}
    Let $X$ be a compactum, and let $F,G\colon X\to 2^X$ lower semicontinuous functions. Then, $G\circ F\colon X\to 2^X$ is a lower semicontinous function.
\end{theorem}

\begin{proof}
Let $V$ be an open subset of $X$. Since $G$ is lower semicontinuous, 
$G^{-1}(\langle X, V\rangle)$ is open (Proposition~\ref{prophfg6tg}).
Also, by Proposition~\ref{prophfg6tg}, $F^{-1}(\langle X, G^{-1}(\langle X, V\rangle)\rangle)$ 
is open and, as a consequence of this, we only need to show that 
\begin{equation}\label{eqlscr}
      (G\circ F)^{-1}(\langle X, V\rangle)=F^{-1}(\langle X, G^{-1}(\langle X, V\rangle)\rangle).  
\end{equation}
Let $x\in (G\circ F)^{-1}(\langle X, V\rangle)$. Then $G(F(x))\cap V \neq\emptyset$. 
Hence, there exists $y\in F(x)$ such that $G(y)\cap V\neq\emptyset$. Thus, 
$y\in G^{-1}(\langle X, V\rangle)\rangle)$. Since 
$y\in F(x)$, we have that $F(x)\cap G^{-1}(\langle X, V\rangle)\rangle)\neq\emptyset$. 
Therefore, $x\in F^{-1}(\langle X, G^{-1}(\langle X, V\rangle)\rangle)$.
Next, assume that $x\in F^{-1}(\langle X, G^{-1}(\langle X, V\rangle)\rangle)$. Hence, 
$F(x)\cap G^{-1}(\langle X, V\rangle)\rangle)\neq\emptyset$. 
Let $y\in F(x)\cap G^{-1}(\langle X, V\rangle)\rangle)$. 
Then $G(y)\cap V\neq\emptyset$. Since $G(y)\subseteq G(F(x))$, we obtain that 
$G(F(x))\cap V\neq\emptyset$. Thus, 
$x\in (G\circ F)^{-1}(\langle X, V\rangle)$. 
Therefore, we have the equality (\ref{eqlscr}).
\end{proof}

Inductively, using Theorem~\ref{theo8ucpa3}, we have the following theorem.

\begin{theorem}\label{theo9yr34rs}
    Let $X$ be a compactum and let $F\colon X\to 2^X$ be a lower semicontinuous function. Then, $F^{n}\colon X\to 2^X$ is lower semicontinuous for each $n\in\mathbb N$.
\end{theorem}

 Next result follows from Theorems \ref{theo9iuh67} and \ref{theo9yr34rs}.
\begin{corollary}
    Let $X$ be a compactum and let $F\colon X\to 2^X$ be a function. If $F$ is continuous, then $F^{n}\colon X\to 2^X$ is continuous for each $n\in\mathbb N$.
\end{corollary}

\begin{theorem}\label{theo87g5}
Let $X$ be a compactum and let 
$F\colon X\to 2^X$ be an upper 
semicontinuous function. If $z\in X$, then
$\mathcal{O}_F(z)$ is a closed subset of $X^{\mathbb N}$.
\end{theorem}

\begin{proof}
Let $(x_i)_{i=1}^{\infty}\in X^{\N}\setminus \mathcal{O}_F(z)$. If $x_1\neq z$, then $W=\pi_1^{-1}(X\setminus\{z\})$ is an open subset of $X^{\N}\setminus \mathcal{O}_F(z)$ and $(x_i)_{i=1}^{\infty}\in W$. Hence, suppose that $x_1=z$. 
Then there exists $k\in \N$ such that $x_{k+1}\notin F(x_k)$. Since $X$ is a regular space and $F$ is upper semicontinuous, there exist open subsets $O, U$ and $V$ of $X$ where $x_k\in O$, $F(x_k)\subseteq U$, $x_{k+1}\in V$, $U\cap V=\emptyset$, and $F(y)\subseteq U$ for all $y\in O$. 
Let $S=\pi_k^{-1}(O)\cap\pi_{k+1}^{-1}(V)$. Note that $S$ is an open subset of $X^{\N}\setminus \mathcal{O}_F(z)$ and $(x_i)_{i=1}^{\infty}\in S$. 
Therefore,  $\mathcal{O}_F(z)$ is a closed subset of $X^{\mathbb N}$. 
\end{proof}

Let $F,G\colon [0,1]\to 2^{[0,1]}$ be given, for each $t\in [0,1]$, by:
$$F(t)=\begin{cases}
    \{0\}, &\text{ if }t=0;\\
    [0,1], &\text{ if }t\in (0,1],
\end{cases}\ \text{ and }\ G(t)=\begin{cases}
    \{0\}, &\text{ if }t=0;\\
    [t,1], &\text{ if }t\in (0,1].
\end{cases}$$
Observe that $(1,0,1,0,0,\ldots)\in \Cl(\mathcal{O}_F(1))\setminus \mathcal{O}_F(1)$, because $$\lim_{n\to\infty}(1,\tfrac{1}{n},1,0,0,\ldots)=(1,0,1,0,0,\ldots),$$ where $(1,\frac{1}{n},1,0,0,\ldots)\in \mathcal{O}_F(1)$ for all $n\in\N$. Thus, $\mathcal{O}_F(1)$ is not closed. Also, we have that 
$$\mathcal{O}_G(t)=\begin{cases}
    \{(0,0,\ldots)\}, &\text{ if }t=0;\\
    \{(t_1,t_2,\ldots) : t_1=t \text{ and }t_i\leq t_{i+1} \text{ for each } i\}, &\text{ if }t\in (0,1].
\end{cases}$$
Thus, $\mathcal{O}_G(t)$ is closed for each $t\in [0,1]$. The functions $F$ and $G$ are not upper semicontinuous.

\begin{theorem}\label{theogf76hy3m}
    Let $X$ be a compactum and let $F\colon X\to 2^X$ be an upper semicontinuous function. Then, $$\bigcup\{\mathcal{O}_F(x) : x\in X\}$$ is compact.
\end{theorem}

\begin{proof}
      Let $(z^n)_{n=1}^{\infty}$ be a sequence in $\bigcup\{\mathcal{O}_F(x) : x\in X\}$ 
    such that $\lim_{n\to\infty}z^n=z$, for some $z\in X^{\mathbb N}$. We prove that 
    $z\in \bigcup\{\mathcal{O}_F(x) : x\in X\}$. Note that if there exists $x_0\in X$ 
    such that $\{n\in \mathbb N : z^n\in \mathcal{O}_F(x_0)\}$ is infinite, then, by Theorem~\ref{theo87g5},  
    $z\in \mathcal{O}_F(x_0)$. Hence, we assume 
    that $z^n\in \mathcal{O}_F(x_1^n)$, for each $n\in\mathbb N$, where 
    $x_1^i\neq x_1^j$ whenever $i\neq j$. Let $z^n=(z_1^n,z_2^n,\ldots)$, 
    for every $n\in\mathbb N$ and let $z=(z_1, z_2,\ldots)$. Since 
    $z_1^n=x_1^n$ and $\lim_{n\to\infty}z^n=z$, we have that $\lim_{n\to \infty}x_1^n=z_1$. 
    We show that $z\in \mathcal{O}_F(z_1)$. Let $m\in\mathbb N$. We 
    know that $\lim_{n\to\infty}z_m^n=z_m$. By Proposition~\ref{prop7uh}, 
    $\limsup F(z_m^n)\subseteq F(z_m)$. Since $z_{m+1}^n\in F(z_m^n)$, 
    for all $n\in\mathbb N$, and $\lim_{n\to\infty}z_{m+1}^n=z_{m+1}$, 
    we have that $z_{m+1}\in \limsup F(z_m^n)$. Thus, $z_{m+1}\in F(z_m)$. 
    Therefore, $z\in \mathcal{O}_F(z_1)\subseteq \bigcup\{\mathcal{O}_F(x) : x\in X\}$.
\end{proof}

Given an upper semicontinuous function $F\colon X\to 2^X$, $x\in X$ and $n\in\mathbb N$, we use the following notation: 
$$\mathcal{O}_F(x)|_{[n]}=\{(y_i)_{i=1}^{\infty}\in X^{\mathbb N} : y_1=x \text{ and }y_{i+1}\in F(y_i) \text{ for each }i\leq n-1\}.$$

\begin{remark}\label{rem001}
Observe that:
\begin{itemize}
    \item $\mathcal{O}_F(x)|_{[n+1]}\subseteq \mathcal{O}_F(x)|_{[n]}$, for each $n\in\mathbb N$; and
    \item $\mathcal{O}_F(x)=\bigcap_{n\in\mathbb N}\mathcal{O}_F(x)|_{[n]}$.
\end{itemize}     
Also, note that, by applying the same argument as in the proof of Theorem~\ref{theo87g5}, we obtain that $\mathcal{O}_F(x)|_{[n]}$ is closed, for each $x\in X$ and all $n\in\mathbb N$. 
\end{remark}

\begin{lemma}\label{lemma0}
    Let $X$ be a continuum and let $F\colon X\to 2^X$ be an upper semicontinuous function. If $F(x)$ is connected, for every $x\in X$, then $\mathcal{O}_F(x)|_{[n]}$ is a continuum, for each $x\in X$ and all $n\in\mathbb N$.
\end{lemma}

\begin{proof}
    By Remark~\ref{rem001}, $\mathcal{O}_F(x)|_{[n]}$ is closed, for each $x\in X$ and all $n\in\mathbb N$. 
    We prove the lemma by induction. Given $x\in X$, $\mathcal{O}_F(x)|_{[1]}=\{(y_i)_{i=1}^{\infty}\in X^{\mathbb N} : y_1=x \}$. 
    Observe that $\mathcal{O}_F(x)|_{[1]}$ is homeomorphic to $X^{\mathbb N}$. Hence, $\mathcal{O}_F(x)|_{[1]}$ is a continuum, for every $x\in X$. 
    Assume that $\mathcal{O}_F(x)|_{[k]}$ is a continuum, for each $x\in X$ and some $k\in\mathbb N$. 
    We see that $\mathcal{O}_F(x)|_{[k+1]}$ is a continuum for all $x\in X$. Let $x\in X$. For each $z\in F(x)$, let $$Y_z=\{(y_i)_{i=1}^{\infty}\in X^{\mathbb N} : y_1=x, y_2=z, \text{ and }y_{i+1}\in F(y_i)\text{ for each }i\in\{2,\ldots,k\}\}.$$ Note that:
    \begin{enumerate}
        \item $\mathcal{O}_F(x)|_{[k+1]}=\bigcup_{y\in F(x)}Y_z$; and
        \item $\phi\colon Y_z\to \mathcal{O}_F(y)|_{[k]}$ given by $\phi((y_i)_{i=1}^{\infty})=(y_i)_{i=2}^{\infty}$ is a homeomorphism; i.e., $Y_z$ is homeomorphic to $\mathcal{O}_F(y)|_{[k]}$, for all $y\in F(x).$
    \end{enumerate}
    Thus, $Y_z$ is a continuum for each $z\in F(x)$.

\bigskip

    Suppose that $\mathcal{O}_F(x)|_{[k+1]}$ is not connected; i.e., there exist two closed sets nonempty and disjoint $\mathcal{M}$ and $\mathcal{N}$ such that $\mathcal{O}_F(x)|_{[k+1]}=\mathcal{M}\cup\mathcal{N}$. 
    Since $Y_z$ is connected, we have $Y_z\subseteq \mathcal{M}$ or $Y_z\subseteq \mathcal{N}$, for each $z\in F(x)$. 
    Let $L_{\mathcal{M}}=\{z\in F(x) : Y_z\subseteq \mathcal{M}\}$ and $L_{\mathcal{N}}=\{z\in F(x) : Y_z\subseteq \mathcal{N}\}$. 
    We prove that $L_{\mathcal{M}}$ is closed.
    Let $(z_n)_{n=1}^{\infty}$ be a sequence in $L_{\mathcal{M}}$ such that $\lim_{n\to\infty}z_n=z$, for some $z\in F(x)$. 
    For each $n\in\mathbb N$, let $w_n\in Y_{z_n}$. 
    Since $(w_n)_{n=1}^{\infty}$ is a sequence in $\mathcal{O}_F(x)|_{[k+1]}$ and $\mathcal{O}_F(x)|_{[k+1]}$ is compact, without loss of generality, we may assume that $\lim_{n\to\infty}w_n=w$, for some $w\in \mathcal{O}_F(x)|_{[k+1]}$. 
    We have $w\in Y_z$. 
    Since $Y_{z_n}\subseteq \mathcal{M}$, $(w_n)_{n=1}^{\infty}$ is a sequence in $\mathcal{M}$. 
    Since $\mathcal{M}$ is closed, we have that
    $w\in \mathcal{M}$ and $Y_z\subseteq \mathcal{M}$. 
    Therefore, $z\in L_{\mathcal{M}}$ and $L_{\mathcal{M}}$ is closed. 
    Similarly, we obtain that $L_{\mathcal{N}} $ is a closed set. 
    Since $\mathcal{M}\cap\mathcal{N}=\emptyset$, $L_{\mathcal{M}}\cap L_{\mathcal{N}}=\emptyset$. 
    Also, we have that $F(x)=L_{\mathcal{M}}\cup L_{\mathcal{N}}$.
    This contradicts the fact that $F(x)$ is connected. 
    We conclude that $\mathcal{O}_F(x)|_{[n]}$ is a continuum for each $x\in X$ and all $n\in\mathbb N$.  
\end{proof}

Since the nested intersection of continua is a continuum \cite[Theorem 1.7.2]{MT}, the next theorem follows from Remark~\ref{rem001} and Lemma~\ref{lemma0}.

\begin{theorem}\label{theom}
    Let $X$ be a continuum and let $F\colon X\to 2^X$ be an upper semicontinuous function. If $F(x)$ is connected for every $x\in X$, then $\mathcal{O}_F(x)$ is a continuum for each $x\in X$.
\end{theorem}

In the following result, we denote by $d$ a metric on $X$ such that $d(x,x')\leq 1$ for each $x,x'\in X$, 
and $\rho$ the metric on $X^{\N}$ defined for each $(y_i)_{i=1}^{\infty},(z_i)_{i=1}^{\infty}\in X^{\N}$, by
\begin{equation}\label{eqmet}
    \rho((y_i)_{i=1}^{\infty},(z_i)_{i=1}^{\infty})=\sum_{i=1}^{\infty}\frac{d(y_i,z_i)}{2^i}.
\end{equation}

\begin{theorem}\label{CantorOrb}
    Let $X$ be a continuum and let $F\colon X\to 2^{X}$ be an upper semicontinuous function. If $F(x)$ is a finite and nondegenerate subset of $X$, for each $x\in X$, then $\mathcal{O}_F(z)$ is a Cantor set for every $z\in X$.
\end{theorem}

\begin{proof}
    Let $z\in X$. By Theorem~\ref{theo87g5}, $\mathcal{O}_F(z)$ is closed. We show that $\mathcal{O}_F(z)$ is perfect. Let $(x_i)_{i=1}^{\infty}\in \mathcal{O}_F(z)$, let $\varepsilon>0$, and let $N\in\N$ be such that $$\sum_{i=N+1}^{\infty}\frac{1}{2^i}<\varepsilon.$$ Since $F(x_N)$ is nondegenerate, there exists $y\in F(x_N)\setminus\{x_{N+1}\}$. Inductively, we construct a sequence $(z_i)_{i=1}^{\infty}\in \mathcal{O}_F(z)$ such that $z_i=x_i$ for each $i\in\{1,\ldots,N\}$ and $z_{N+1}=y$. Hence, $(z_i)_{i=1}^{\infty}\neq (x_i)_{i=1}^{\infty}$ and by (\ref{eqmet}), $$\rho((z_i)_{i=1}^{\infty},(x_i)_{i=1}^{\infty})= \sum_{i=N+1}^{\infty}\frac{d(z_i,x_i)}{2^i}\leq \sum_{i=N+1}^{\infty}\frac{1}{2^i}<\varepsilon.$$
    Thus, $\mathcal{O}_F(z)$ is perfect.

    Finally, since $F^k(z)$ is finite, we have that $\pi_k(\mathcal{O}_F(z))$ is finite  for all $k\in\N$, by Lemma~\ref{lem0}. Thus, $\mathcal{O}_F(z)$ is a totally disconnected subset of $X^{\N}$. Therefore, $\mathcal{O}_F(z)$ is a Cantor set \cite[Theorem~30.3]{Will}.
\end{proof}

Let $F\colon [0,1]\to 2^{[0,1]}$ be given by 
$F(t)=\{t,1-t\}$. Observe that $F(\frac{1}{2})=
\{\frac{1}{2}\}$; i.e., $\mathcal{O}_F(\frac{1}{2})=\{(\frac{1}{2}, \frac{1}{2}, \ldots )\}$ is not a Cantor set. However, if $t\in [0,1]\setminus\{{1\over 2}\}$, then there exists a bijection (a homeomorphism) between the all the sequences of $0$'s and $1$'s and $\mathcal{O}_F(t)$. Hence, $\mathcal{O}_F(t)$ is a Cantor set.


\begin{example}
    Let $f\colon [0,1]\tto [0,1]$ be the devil's staircase 
function \cite[Figure~3-19, p.~131]{HY} and let
$F\colon [0,1]\to 2^{[0,1]}$ be given by $F(t)=\{t,f(t)\}$.
Note that $F$ is continuous, and $F(t)=\{t\}$ if and only if $t\in\{0,\frac{1}{2},1\}$.
If $t\in [\frac{1}{3},\frac{2}{3}]$, then 
$$\mathcal{O}_F(t)=\left\{(t_1,\ldots, t_n,\tfrac{1}{2},\tfrac{1}{2},\ldots): t_1=\ldots=t_n=t, n\in\mathbb N\right\}\cup
\left\{(t,t,\ldots)\right\}.$$
Hence, even though for each 
$t\in [0,1]\setminus\left\{0,\frac{1}{2},1\right\}$,
$F(t)$ consists of two points, we have many orbits 
of points of $[0,1]$, with respect to $F$, that are 
not Cantor sets.
\end{example}


\begin{question}
    Let $F\colon X\to 2^X$ be an upper semicontinuous function. If $F(x)$ is a nondegenerate and totally disconnected subset of $X$, for each $x\in X$, then does it follow that $\mathcal{O}_F(z)$ is a Cantor set, for every $z\in X$?
\end{question}

\section{More about orbit sets}\label{MoreOrbits}

We know that if $X$ is a continuum and $F\colon X\to 2^X$ is an upper semicontinuous function such that $F(x)$ is connected, for each $x\in X$, then $\mathcal{O}_F(x)$ is a continuum for every $x\in X$ (see Theorem~\ref{theom}). Given a fixed continuum $X$, it is interesting to explore the types of continua that can be obtained using various upper semicontinuous functions. In this section, we provide examples of continua when $X$ is an arc.

\begin{question}\label{quest0}
Let $X$ be a compactum. Does there exist an upper semicontinuous function $F\colon X\to 2^{X}$ and
$p\in X$ such that $\mathcal{O}_F(p)$ is homeomorphic to $X$?
\end{question}

We do not currently have an answer to this question. The following example presents some positive answers to Question~\ref{quest0}.

\begin{example}
    For each compactum $X$, we define an upper semicontinuous function $F\colon X\to 2^X$ and a point $p\in X$ such that $\mathcal{O}_F(p)$ is homeomorphic to $X$.

    \begin{enumerate}
        \item Let $X=\{\frac{1}{n} : n\in\mathbb N\}\cup\{0\}$ and let $F\colon X\to 2^X$ be given by $F(x)=\{x\}$, for each $x\in X\setminus \{1\}$, and $F(1)=X\setminus\{1\}$. Note that $F$ is upper semicontinuous and $\mathcal{O}_F(1)$ is homeomorphic to $X\setminus\{1\}$. Thus, $\mathcal{O}_F(1)$ is homeomorphic to $X$.
        
        \item Let $X=\{0,1\}^{\mathbb N}$, let $\mathcal{C}_0=\{(x_i)_{i=1}^{\infty}\in X : x_1=0\}$ and let $\mathcal{C}_1=\{(x_i)_{i=1}^{\infty}\in X : x_1=1\}$. We define $F\colon X\to 2^X$ by $F(x)=\{x\}$ if $x\in \mathcal{C}_0$ and $F(x)=\mathcal{C}_0$ whenever $x\in \mathcal{C}_1$. Note that we have that $\mathcal{O}_F(p)$ is homeomorphic to $\mathcal{C}_0$ where $p\in \mathcal{C}_1$. Since  $\mathcal{C}_0$ is homeomorphic to $X$, $\mathcal{O}_F(p)$ is homeomorphic to $X$.

\bigskip

The previous examples illustrate the same idea in constructing the upper semicontinuous function. These examples can be viewed as a particular case of the following example.

\bigskip

        \item Let $X$ be a compactum. Suppose that $X=M\cup N$, where $M$ and $N$ are nonempty and disjoint closed subsets of $X$ such that there exists $Y\subseteq M$ homeomorphic to $X$. Let $F\colon X\to 2^X$ be given by 
        $$F(x)=\begin{cases}
            Y, &\text{ if }x\in N;\\
            \{x\}, &\text{ if }x\in M.
        \end{cases}$$
        Let $p\in N$. Then $\mathcal{O}_F(p)=\{(p,y,y,\ldots) : y\in Y\}$ is homeomorphic to $Y$. Therefore, $\mathcal{O}_F(p)$ is homeomorphic to $X$.
        
        \item Let $X=[0,1]$ and let $F\colon X\to 2^X$ given such that $F(0)=[\frac{1}{2},1]$, $F(t)=\{t\}$, for each $t\in [\frac{1}{2},1]$, and $\mathrm{Gr}(F)$ is pictured in Figure~\ref{fig:enter-labelegt34}. Observe that $\mathcal{O}_F(0)=\{(0,t,t,\ldots) : t\in [\frac{1}{2},1]\}$ is homeomorphic to $[0,1]$.
         \begin{figure}[h]
    \centering
\begin{tikzpicture}[scale=1]
\draw[->,very thin,gray] (-1,0) to (2.5,0);
\draw[->,very thin,gray] (0,-1) to (0,2.5);
\draw (0,2) -- (0,1) -- (0.5,0) -- (1,1) -- (2,2);
\draw[dotted] (0,2) -- (2,2) --  (2,0);
\end{tikzpicture}
    \caption{$\mathrm{Gr}(F)$.}
    \label{fig:enter-labelegt34}
\end{figure}
    \end{enumerate}
\end{example}

We present an example in $[0,1]$ that illustrates an infinite-dimensional continuum.

\begin{proposition}\label{ex0}
There exists an upper semicontinuous function $F\colon [0,1]\to 2^{[0,1]}$ such that $\mathcal{O}_F(t)$ is an infinite-dimensional continuum.    
\end{proposition}

\begin{proof}
        Let $r\in [0,1]$ and let $F\colon [0,1]\to 2^{[0,1]}$ be given by $$F(t)=\begin{cases}
        [0,1],& \text{ if }t=r;\\
        \{r\},& \text{ if }t\in [0,1]\setminus \{r\}.
    \end{cases}$$
    
    Let $t_0\in [0,1]$ and let 
    $\varphi\colon [0,1]^{\mathbb N}\to \mathcal{O}_F(t_0)$ be given for each $(s_n)_{n=1}^{\infty}\in [0,1]^{\mathbb N}$, by 
    $$\varphi ((s_n)_{n=1}^{\infty})=(t_0,r,s_1,r,s_2,r,s_3,r,\ldots).$$ 
 Observe that $\varphi$ is an embedding. Thus, $\mathrm{dim}(\mathcal{O}_F(t_0))=\infty$. Since $F(t)$ is connected for each $t\in [0,1]$, $\mathcal{O}_F(t_0)$ is a continuum, by  Theorem~\ref{theom}.
    
\end{proof}

Now, we present examples of dendrites that can be represented as an orbit set of the arc.

\begin{example}\label{exm8uh6}
    Let $F\colon [0,1]\to 2^{[0,1]}$ be defined by $$F(t)=\begin{cases}
        [0,1], &\text{ if }t=0;\\
        \{t\}, &\text{ if }t\neq 0.
    \end{cases}$$
    Then $\mathcal{O}_F(0)$ is homeomorphic to $F_{\omega}$, where $F_{\omega}$ is the infinite union of segments joining $(0,0)$ with $(\frac{1}{n},\frac{1}{n^2})$, $n\in\mathbb N$, as represented in  Figure~\ref{fig:enter-labelegt}.
    \begin{figure}[h]
    \centering
\begin{tikzpicture}[scale=3]
\draw (0,0) to (1,1);
\draw (0,0) to (0.5,0.25);
\draw (0,0) to (0.3,0.09);
\draw[dotted] (0.12,0) to (0.19,0.03);
\end{tikzpicture}
    \caption{Dendrite $F_{\omega}$.}
    \label{fig:enter-labelegt}
\end{figure}
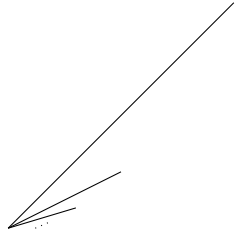

For each $n\in\mathbb N$, let 
\begin{multline*}
  A_n=\{(t_1,t_2,\ldots)\in [0,1]^{\mathbb N} : t_1=\cdots=t_{n}=0 \text{ and }\\ t_i=s, \text{ for each }i\geq n+1 \text{ and some }s\in (0,1]\}. 
\end{multline*}

Observe that $A_i\cap A_j=\emptyset$ whenever $i\neq j$, $A_n\cong (0,1]$ and $\overline{A_n}\setminus A_n=\{(0,0,\ldots)\}$, for every $n\in\mathbb N$.
Thus, $\overline{A_n}$ is an arc ($\overline{A_n}$ is homeomorphic to $[0,1]$) for each $n\in \mathbb N$. 
Furthermore, $\mathcal{O}_F(0)=\left(\bigcup_{n\in\mathbb N}A_n\right)\bigcup\{(0,0,\ldots)\}$ and $\lim_{n\to\infty} \mathrm{diam}(A_n)=0$. Therefore, $\mathcal{O}_F(0)$ is the dendrite $F_{\omega}$ represented in Figure~\ref{fig:enter-labelegt}.

\end{example}
\begin{example}\label{exoif6hdt7}
    Let $F\colon [0,1]\to 2^{[0,1]}$ be given by $$F(t)=\begin{cases}
        [0,1], &\text{ if }t\in\{0,1\};\\
        \{t\}, &\text{ if }t\in (0,1).
    \end{cases}$$
\end{example}
\noindent We describe the orbit of ``$0$'', $\mathcal{O}_F(0)$. Let $A_{\emptyset}=\{(0,t,t,\ldots) : t\in (0,1)\}$. We have that $A_{\emptyset}\subseteq \mathcal{O}_F(0)$ and $\overline{A_{\emptyset}}=A_{\emptyset}\cup\{(0,0,0,\ldots), (0,1,1,\ldots)\}$. Moreover, $\overline{A_{\emptyset}}$ is homeomorphic to $[0,1]$.

For each $(a_1,\ldots,a_k)\in \{0,1\}^k$, let 
$$A_{a_1\cdots a_k}=\{(0,a_1,\ldots,a_k,t,t,\ldots) : t\in (0,1)\}.$$
Observe that $A_{a_1\cdots a_k}\subseteq \mathcal{O}_F(0)$,
$$\overline{A_{a_1\cdots a_k}}=A_{a_1\cdots a_k}\cup\{(0,a_1,\ldots,a_k,0,0,\ldots),(0,a_1,\ldots,a_k,1,1,\ldots)\},$$ 
and $\overline{A_{a_1\cdots a_k}}$ is homeomorphic to $[0,1]$, for each $(a_1,\ldots,a_k)\in \{0,1\}^k$. 
Note that we have the following:
\begin{itemize}
    \item $A_{a_1\cdots a_k}\cap A_{b_1\cdots b_l}=\emptyset$, whenever $(a_1,\ldots,a_k)\neq (b_1,\ldots,b_l)$.
    \item Given $(a_1,\ldots,a_k)\in \{0,1\}^k$,
    \begin{multline*}
        \overline{A_{a_1\cdots a_k}}\cap \overline{A_{a_1\cdots a_k0}}=\{(0,a_1,\ldots,a_k,0,0,\ldots)\} \text{ and }\\ 
    \overline{A_{a_1\cdots a_k}}\cap \overline{A_{a_1\cdots a_k1}}=\{(0,a_1,\ldots,a_k,1,1,\ldots)\}.
    \end{multline*}
    \item $\mathrm{diam}(\overline{A_{a_1\cdots a_k}})=\frac{1}{2^k}$ for each $(a_1,\ldots,a_k)\in \{0,1\}^k$ (use (\ref{eqmet})).
    \item $\mathcal{O}_F(0)=\left(\bigcup\{A_{a_1\cdots a_k} : (a_1,\ldots,a_k)\in \{0,1\}^k, k\in\mathbb N\}\right)\cup A_{\emptyset}\cup \{0,1\}^{\mathbb N}$.
\end{itemize}
By Theorem~\ref{theom}, $\mathcal{O}_F(0)$ is a continuum. We represent the orbit set $\mathcal{O}_F(0)$ in Figure~\ref{figCantor}.
\begin{figure}[h]
    \centering
\begin{tikzpicture}[scale=0.8]
\draw (1,0) arc (0:180:0.5cm);
\draw (3,0) arc (0:180:0.5cm);
\draw (7,0) arc (0:180:0.5cm);
\draw (9,0) arc (0:180:0.5cm);
\draw (9,0) arc (0:180:1.5cm);
\draw (3,0) arc (0:180:1.5cm);

\draw (0.333,0) arc (0:180:0.1665cm);
\draw (2.333,0) arc (0:180:0.1665cm);
\draw (6.333,0) arc (0:180:0.1665cm);
\draw (8.333,0) arc (0:180:0.1665cm);

\draw (1,0) arc (0:180:0.1665cm);
\draw (3,0) arc (0:180:0.1665cm);
\draw (7,0) arc (0:180:0.1665cm);
\draw (9,0) arc (0:180:0.1665cm);

\draw (9,0) arc (0:180:4.5cm);

\draw[very thin,dotted] (0,0)-- (0.333,0);
\draw[very thin,dotted] (0.666,0)-- (1,0);

\draw[very thin,dotted] (2,0)-- (2.333,0);
\draw[very thin,dotted] (2.666,0)-- (3,0);

\draw[very thin,dotted] (6,0)-- (6.333,0);
\draw[very thin,dotted] (6.666,0)-- (7,0);

\draw[very thin,dotted] (8,0)-- (8.333,0);
\draw[very thin,dotted] (8.666,0)-- (9,0);

\draw (12.3,0) node {$\{0,1\}^{\mathbb N}$};
\draw (4.5,5.2) node {\small{$A_{\emptyset}$}};
\draw (-0.5,-0.5) node {\small{$(0,0,0,\ldots)$}};
\draw (9.5,-0.5) node {\small{$(0,1,1,\ldots)$}};
\draw (1.5,2) node {\small{$A_{0}$}};
\draw (7.5,2) node {\small{$A_{1}$}};
\draw (3.2,-0.5) node {\small{$(0,0,1,\ldots)$}};
\draw (5.8,-0.5) node {\small{$(0,1,0,\ldots)$}};
\draw (0,0) node {\tiny{$\bullet$}};
\draw (3,0) node {\tiny{$\bullet$}};
\draw (6,0) node {\tiny{$\bullet$}};
\draw (9,0) node {\tiny{$\bullet$}};
\draw (0.7,0.8) node {\small{$A_{00}$}};
\draw (2.3,0.8) node {\small{$A_{01}$}};
\draw (6.7,0.8) node {\small{$A_{10}$}};
\draw (8.3,0.8) node {\small{$A_{11}$}};
\end{tikzpicture}
    \caption{Continuum $\mathcal{O}_F(0)$.}
    \label{figCantor}
\end{figure}
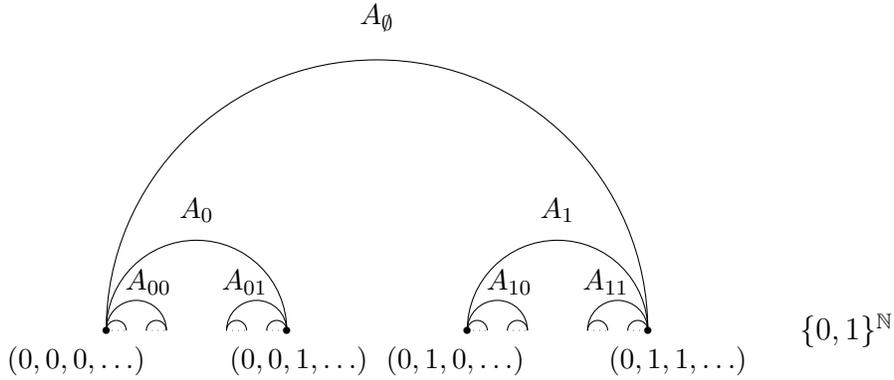
\medskip

Let $\mathcal{I}=\{Z:Z \text{ is a continuum and 
there exist }F\colon [0,1]\to 2^{[0,1]} \text{ and }t_0\in [0,1],\break \text{ where }\mathcal{O}_F(t_0)\cong Z \}$.

\begin{problem}
    Determine which continua $Z$ belong to $\mathcal{I}$.
\end{problem}

Now, we present some examples using the set function $\T$. Note that if $X$ is an aposyndetic continuum, then $\T(\{x\})=\{x\}$ for each $x\in X$ \cite[Theorem~3.1.28]{MT}. Hence, if $T\colon X\to 2^X$ is given as a Remark~\ref{defTT}, we see that $\mathcal{O}_T(x)$ is a single point for all $x\in X$.

\begin{remark}
    Let $X$ be an indecomposable continuum. Let $T\colon X\to 2^X$ be defined by $T(x)=\T(\{x\})$ for each $x\in X$. By \cite[Theorem~3.1.39]{MT}, $T(x)=X$ for each $x\in X$. Thus, $\mathcal{O}_T(x)$ is homeomorphic to $X^\mathbb N$ for each $x\in X$.
\end{remark}

\begin{example}
Let $X$ be the harmonic fan; that is, $X=\bigcup_{i=0}^{\infty} \overline{oa_i}$, where $a_i=(1,\frac{1}{i})$ whenever $i\neq 0$, $a_0=(1,0)$, $o=(0,0)$ and $\overline{ab}=\{at+(1-t)b : t\in [0,1]\}$ for each $a,b\in \mathbb R^2$. Let $T\colon X\to 2^X$ be defined by $T(x)=\T(\{x\})$ (Remark~\ref{defTT}). Observe that $$\mathcal{O}_T(o)=\{(o,x_2,x_3,\ldots): x_{n+1}\in \overline{x_na_0} \text{ for each }n\geq 2\}.$$ Therefore, $\mathcal{O}_T(o)$ is homeomorphic to the Hlbert cube $[0,1]^{\mathbb N}$.   
\end{example}

\begin{example}\label{Xhomogcont}
Let $X$ be a decomposable homogeneous continuum, and let
$T\colon X\to 2^X$ be given by $T(x)=\T(\{x\})$, for each $x\in X$. By
\cite[Theorem 5.1.18]{MT}, $\G=\{\T(\{x\})\ |\ x\in X\}$ is a 
continuous decomposition into continua. Let $x\in X$. Then, by definition,
$T^2(x)=\bigcup\{T(z)\ |\ z\in T(x)\}=\bigcup\{\T(\{z\})\ |\ z\in\T(\{x\})\}$.
Let $z\in\T(\{x\})$. Note that, by \cite[Proposition 2.1.7]{MS}, 
$\T(\{z\})\subset\T^2(\{x\})$. Since $\T$ is idempotent on closed sets 
\cite[Theorem 4.2.32]{MT}, we have that $\T^2(\{x\})=\T(\{x\})$. Hence,
$\T(\{z\})\subset\T(\{x\})$. Thus, since $\G$ is a decomposition,
we obtain that $\T(\{z\})=\T(\{x\})$. Therefore, we have that
$T^2(x)=\bigcup\{\T(\{z\})=\T(\{x\})\ |\ z\in\T(\{x\})\}=\T(\{x\})=T(x)$,
for all $x\in X$. Hence, for every $x\in X$, 
$\OO_T(x)=\{x\}\times\prod_{n=1}^\infty\T(\{x\})_n$, where $\T(\{x\})_n=\T(\{x\})$,
for each positive integer $n$. 
Therefore, the orbit sets of $T$ are continua. 
\end{example}

Given an onto map $f\colon X\to X$ where $X$ is a compactum, the map $F\colon X\to 2^X$ given by $F(x)=f^{-1}(x)$, for each $x\in X$, is upper semicontinuous \cite[Exercise~7.15$(b)$]{N2}. 
Observe that if $x_0\in X$, then 
$$\mathcal{O}_F(x_0)=\{(z_n)_{n=1}^{\infty}\in X^{\mathbb N} : z_1=x_0 \text{ and } z_n=f(z_{n+1}) \text{ for all }n\in\mathbb N\}.$$
Thus, 
$$\bigcup\{\mathcal{O}_F(x) : x\in X\}=\underleftarrow{\lim} \{X,f\}.$$ 
Therefore, $\bigcup\{\mathcal{O}_F(x) : x\in X\}$ is compact. We generalize this remark with the following theorems.

\begin{theorem}\label{theo8uhu7yh7}
Let $X$ be a compactum and let $f_1,\ldots,f_k\colon X\to X$ onto maps. Let $F\colon X\to 2^X$ be defined by $F(x)=f_1^{-1}(x)\cup\cdots\cup f_k^{-1}(x)$, for every $x\in X$. For each $p=(p_1, p_2, \ldots)\in \{1,\ldots,k\}^{\mathbb N}$, we define
\begin{equation}\label{eq00o0}
    X_p=\underleftarrow{\lim}\{X,f_{p_n}\}_{n\in\mathbb N}.
\end{equation}
Then  
\begin{equation}\label{eq13}
    \bigcup\{\mathcal{O}_F(x) : x\in X\}=\bigcup \{X_p : p\in \{1,\ldots,k\}^{\mathbb N}\}.
\end{equation}    
\end{theorem}

\begin{proof}
Note that a finite union of upper semicontinuous functions is an upper semicontinuous function. Hence, $F$ is an upper semicontinuous.
Let $p=(p_1, p_2, \ldots)\in \{1,\ldots,k\}^{\mathbb N}$. 
Observe that if $(x_n)_{n=1}^{\infty}\in X_p$, then $x_n=f_{p_n}(x_{n+1})$,  
for each $n\in\mathbb N$. Hence, $x_{n+1}\in f_{p_n}^{-1}(x_n)\subseteq F(x_n)$. 
Thus, $(x_n)_{n=1}^{\infty}\in \mathcal{O}_F(x_1)$, and we have that 
$\bigcup \{X_p : p\in \{1,\ldots,k\}^{\mathbb N}\}\subseteq \bigcup\{\mathcal{O}_F(x) : x\in X\}$.

Conversely, let $(x_n)_{n=1}^{\infty}\in \mathcal{O}_F(x_1)$. 
Given $n\in \mathbb N$, let $q_n\in\{1,\ldots,k\}$ be such that 
$x_{n+1}\in f_{q_n}^{-1}(x_n)$. Hence, $(x_n)_{n=1}^{\infty}\in X_q$, 
where $q=(q_1, q_2,\ldots)$. Thus, 
$\bigcup\{\mathcal{O}_F(x) : x\in X\}\subseteq \bigcup \{X_p : p\in \{1,\ldots,k\}^{\mathbb N}\}.$ 
Therefore, $\bigcup\{\mathcal{O}_F(x) : x\in X\}=\bigcup \{X_p : p\in \{1,\ldots,k\}^{\mathbb N}\}.$
\end{proof}

\begin{theorem}\label{theoh5h5h5h}
    Let $X$ be a continuum and let $f_1,\ldots,f_k\colon X\to X$ be onto maps such that there exists $x_0\in X$ for which $f_1^{-1}(x_0)\cap\cdots\cap f_k^{-1}(x_0)\neq\emptyset$.
    If $F\colon X\to 2^X$ is defined by $F(x)=f_1^{-1}(x)\cup\cdots\cup f_k^{-1}(x)$, for each $x\in X$, then $$\bigcup\{\mathcal{O}_F(x) : x\in X\}$$ is a continuum.
\end{theorem}

\begin{proof}
    By Theorem~\ref{theogf76hy3m}, $\bigcup\{\mathcal{O}_F(x) : x\in X\}$ is compact. We prove that $\bigcup\{\mathcal{O}_F(x) : x\in X\}$ is connected. Let $(z_n)_{n=1}^{\infty}$ and $(w_n)_{n=1}^{\infty}$ be points in $\bigcup\{\mathcal{O}_F(x) : x\in X\}$.

    \bigskip

    \textbf{Claim.} Let $p=(p_1, p_2, \ldots)$ and $q=(q_1, q_2, \ldots)$ be points in $\{1,\ldots,k\}^{\mathbb N}$ such that there exists $m\in\mathbb N$ where $p_i=q_i$ for each $i\neq m$. Then $X_p\cap X_q\neq\emptyset$,  where $X_p$ and $X_q$ are defined as in equation (\ref{eq00o0}).

    \medskip

    Let $y\in f_1^{-1}(x_0)\cap\cdots\cap f_k^{-1}(x_0)$. 
    Since $f_1, \ldots, f_k$ are onto maps, there exists $(x_n)_{n=1}^{\infty}\in X_p$ such that $x_m=y$. Since $y\in f_1^{-1}(x_0)\cap\cdots\cap f_k^{-1}(x_0)$, we have that $f_{q_m}(y)=f_{p_m}(y)$. Hence, $(x_n)_{n=1}^{\infty}\in X_q$ and $X_p\cap X_q\neq\emptyset$. We proved the Claim.

    \bigskip

    By Theorem~\ref{theo8uhu7yh7} (\ref{eq13}), let $p=(p_1, p_2, \ldots)$ and $q=(q_1, q_2, \ldots)$ be points in $\{1,\ldots,k\}^{\mathbb N}$ such that $(z_n)_{n=1}^{\infty}\in X_p$ and $(w_n)_{n=1}^{\infty}\in X_q$. For each $m\in \mathbb N$, let $q^mp=(q_1,\ldots,q_m,p_{m+1}, p_{m+2},\break\ldots)$. 
    By Claim, we have that $X_p\cap X_{q^1p}\neq\emptyset$, $X_{q^1p}\cap X_{q^2p}\neq\emptyset$, and in general, $X_{q^kp}\cap X_{q^{k+1}p}\neq\emptyset$, for each $k\in\mathbb N$. 
    Since for each $l\in \{1,\ldots,k\}^{\mathbb N}$, $X_l$ is a continuum  \cite[Theorem~2.1.8]{MT}, $X_p\cup\left(\bigcup_{m=1}^{\infty}X_{q^mp}\right)$ is connected. 
    Also, $(w_n)_{n=1}^{\infty}\in \overline{X_p\cup\left(\bigcup_{m=1}^{\infty}X_{q^mp}\right)}$. Hence, $X_p\cup\left(\bigcup_{m=1}^{\infty}X_{q^mp}\right)\cup \{(w_n)_{n=1}^{\infty}\}$ is a connected subset of $\bigcup\{\mathcal{O}_F(x) : x\in X\}$, containing $(z_n)_{n=1}^{\infty}$ and $(w_n)_{n=1}^{\infty}$. Therefore, $\bigcup\{\mathcal{O}_F(x) : x\in X\}$ is connected.
\end{proof}



\begin{corollary}
Let $X$ be an arc-like continuum, and let $f,g\colon X \to X$ be onto maps.
If $F\colon X\to 2^X$ is given by $F(x)=f^{-1}(x)\cup g^{-1}(x)$, for all $x\in X$,
then $\bigcup\{\mathcal{O}_F(x) : x\in X\}$ is a continuum.
\end{corollary}

\begin{proof}
Observe that, by \cite[Corollary 2.6.10]{MT}, there exists $x\in X$ such that $f(x)=g(x)$.
Now, the corollary follows from Theorem~\ref{theoh5h5h5h}.
\end{proof}

\section{Dense orbits and transitivity}\label{DensOrbTrans}

Let $X$ be a compactum and let $F\colon X\to 2^X$ be an upper semicontinuous function and let $p\in X$. We say that:
\begin{enumerate}
    \item $F$ is \textit{transitive} provided for each pair of nonempty open subsets $U$ and $V$ of $X$, there exist $x\in U$ and $k\in\mathbb N$ such that $F^k(x)\cap V\neq\emptyset$.
    
    \item $p$ has a \textit{dense orbit} if there exists an orbit $(x_n)_{n=1}^{\infty}\in \mathcal{O}_F(p)$ such that for all open subset $U$ of $X$, there exists $k\in\mathbb N$ where $x_k\in U$.
    
    \item $p$ has a \textit{weak dense orbit} provided for each open subset $U$ of $X$, there exists $k\in \mathbb N$ such that $F^k(p)\cap U\neq\emptyset$.

\end{enumerate}

\begin{proposition}\label{proptgf6r}
    Let $X$ be a compactum with no isolated points and let $F\colon X\to 2^X$ be an upper semicontinuous function. Given the following:
    \begin{enumerate}
        \item There exists $p\in X$ such that $p$ has a dense orbit;
        
        \item $F$ is transitive;
        
        \item There exists $p\in X$ such that $p$ has a weak dense orbit.
    \end{enumerate}
    Then 1 implies 2, and 1 implies 3. Moreover, if $F$ is continuous, then 2 implies 3.
\end{proposition}

\begin{proof}
    We show that \textit{1} implies \textit{2}. Let $U$ and $V$ be nonempty open subsets of $X$. Let $(x_n)_{n=1}^{\infty}\in \mathcal{O}_F(p)$ be dense in $X$. Let $k\in\mathbb N$ such that $x_k\in U$. Since $X$ has no isolated points, $V\setminus\{x_1,\ldots,x_k\}$ is a nonempty open set. Hence, there exists $l>k$ such that $x_l\in V$. Note that $x_k\in U$, $(x_{k+n-1})_{n=1}^{\infty}\in \mathcal{O}_F(x_k)$, and $x_l\in F^{l-k}(x_k)\cap V$. Thus, $F$ is transitive. Observe that \textit{1} implies \textit{3} follows from definitions.

\medskip
 
    Next, suppose that $F$ is continuous, we prove that \textit{2} implies \textit{3}. Let $\mathcal{B}=\{B_n : n\in \mathbb N\}$ be a countable base of $X$. For each $n\in\mathbb N$, let $W_n=\bigcup_{k\in\mathbb N}F^{-k}(\langle X,B_n\rangle)$. Note that $W_n$ is open for each $n\in\mathbb N$. Also, since $F$ is transitive, $W_n$ is dense for all $n$. Thus, by the Baire Category Theorem \cite[Theorem 25.2]{Will}, there exists $p\in \bigcap_{n\in\mathbb N}W_n$. Observe that if $V$ is an open set, then $B_m\subseteq V$ for some $m$. Since $p\in W_m$, $F^l(p)\in \langle X, B_m\rangle$; i.e., $F^l(p)\cap V\neq\emptyset$. Therefore, $p$ has a weak dense orbit.
\end{proof}

\begin{remark}
Let $F\colon [0,1]\to 2^{[0,1]}$ be the upper semicontinuous function defined for each $t\in [0,1]$ by: $$F(t)=\begin{cases}
    \{t\}, &\text{ if }t\in [0,1);\\ [0,1], &\text{ if }t=1.
\end{cases}$$
Note that $1$ has a weak dense orbit, but $F$ is not transitive (take two nonempty disjoint open subsets
$U$ and $V$ of $[0,1)$ and there do not exist $x\in U$ and $k\in\mathbb N$ such that $F^k(x)\cap V\neq\emptyset$). Thus, there does not exist a point $p$ with dense orbit  
(Theorem~\ref{proptgf6r}). 
\end{remark}

Next example shows a continuous function $F\colon X\to 2^X$ such that $F$ is not transitive and there exists $p\in X$ such that $p$ has a weak dense orbit.

\begin{example}
    Let $X=(\{0, 1, 2\}\times [0,1])\cup ([0,2]\times \{0\})$. Let $F\colon X\to 2^X$ be a continuous extension of $G\colon \{0, 1, 2\}\times [0,1] \to 2^{X}$ given by $$G(x)=\begin{cases}
        \{x\}, &\text{ if }x\in \{0\}\times [0,1];\\
        \{2\}\times [0,1], &\text{ if }t\in \{1\}\times [0,1];\\
        (\{0,1\}\times [0,1])\cup ([0,1]\times \{0\}), &\text{ if }t\in \{2\}\times [0,1].
    \end{cases}$$
    Observe that if $U$ and $V$ are disjoint open subsets of 
    $\{0\}\times [0,1]$, then $F(x)=\{x\}$, for each $x\in U$. Hence, $F^k(x)\cap V=\emptyset$, for all $k\in\mathbb N$. Therefore, $F$ is not transitive.
    Let $x_0\in \{2\}\times [0,1]$. We show that $x_0$ has a weak dense orbit. Let $U$ be an open subset of $X$. We consider three cases:
    \begin{enumerate}
        \item If $U\cap ((\{0,1\}\times [0,1])\cup ([0,1]\times \{0\}))\neq\emptyset$, then $F(x_0)\cap U\neq\emptyset$.
        \item If $U\cap (\{2\}\times [0,1])\neq\emptyset$, since $\{1\}\times [0,1]\subseteq F(x_0)$ and 
        $F(z)=\{2\}\times [0,1]$ for each $z\in 
        \{1\}\times [0,1]$, we have that $\{2\}\times [0,1]\subseteq F^2(x_0)$. Thus, $F^{2}(x_0)\cap U\neq\emptyset$.
        \item If $U\subseteq (1,2)\times \{0\}$, since $F$ is continuous and $\{\{(0,0)\}, (\{2\}\times [0,1])\}\subseteq F([0,1]\times \{0\})$, we obtain that $(1,2)\times \{0\}\subseteq \bigcup F([0,1]\times \{0\})$.
        Also, $[0,1]\times \{0\}\subseteq F(x_0)$. Thus, $(1,2)\times \{0\}\subseteq \bigcup F^2(x_0)$. Therefore, $F^2(x_0)\cap U\neq\emptyset$.
    \end{enumerate}
    We conclude that $x_0$ has a weak dense orbit.
\end{example}

Next example shows that if $F$ is not continuous, then we do not necessarily 
have that \textit{2} implies \textit{3} in Proposition \ref{proptgf6r}.
\begin{example}
    There exists an upper semicontinuous function $F\colon [0,1]\to 2^{[0,1]}$ such that $F$ is transitive, 
    but $x$ does not have a weak dense orbit for any $x\in X$.
\end{example}
Let $h\colon [0,1]\to [0,1]$ be the map pictured in Figure \ref{fig:enter-label}, defined for each $x\in [0,1]$ by, $$h(x)=\begin{cases}
    -2x+\frac{1}{2}, &\text{ if } x\in [0,\frac{1}{4}];\\ 
    2x-\frac{1}{2}, &\text{ if } x\in [\frac{1}{4},\frac{3}{4}];\\
    -2x+\frac{5}{2}, &\text{ if } x\in [\frac{3}{4},1].
\end{cases}$$

\begin{figure}[h]
    \centering
\begin{tikzpicture}[scale=3]
\draw[->] (-0.35,0) -- (1.35,0);
\draw[->] (0,-0.35) -- (0,1.35);
\draw (0,0.5) -- (0.25,0) -- (0.75,1) -- (1,0.5);
\draw [dashed] (0,1) -- (1,1) -- (1,0);
\draw [dashed] (0.5,0) -- (0.5,1);
\draw [dashed] (0,0.5) -- (1,0.5);
\draw (0.5,-0.15) node {$\frac{1}{2}$};
\draw (1,-0.15) node {$1$};
\draw (-0.15,0.5) node {$\frac{1}{2}$};
\draw (-0.15,1) node {$1$};
\end{tikzpicture}
    \caption{Graph of $h\colon [0,1]\to [0,1]$.}
    \label{fig:enter-label}
\end{figure}

Note that both $h|_{[0,\frac{1}{2}]}\colon [0,\frac{1}{2}]\to [0,\frac{1}{2}]$ and 
$h|_{[\frac{1}{2},1]}\colon [\frac{1}{2},1]\to [\frac{1}{2},1]$ are transitive. 
Hence, by Proposition~\ref{propB}, there exist $t\in [0,\frac{1}{2}]$ and $s\in [\frac{1}{2},1]$, with dense orbits on 
$[0,\frac{1}{2}]$ and $[\frac{1}{2},1]$, respectively. Observe that $h(\frac{1}{6})=\frac{1}{6}$ 
and $h(\frac{5}{6})=\frac{5}{6}$.

Let $F\colon [0,1]\to 2^{[0,1]}$ be defined for each $x\in [0,1]$ by:
$$F(x)=\begin{cases}
    \{h(x)\}, &\text{ if }x\in [0,1]\setminus \{\frac{1}{6},\frac{5}{6}\};\\
    \{\frac{1}{6}, s\}, &\text{ if }x=\frac{1}{6};\\
    \{\frac{5}{6}, t\}, &\text{ if }x=\frac{5}{6}.
\end{cases}$$
We show that $F$ is transitive. Let $U$ and $V$ be open subsets of $[0,1]$. 
Without loss of generality, we may assume that $U\cap [0,\frac{1}{2}]\neq\emptyset$.
Let $k\in\mathbb N$ be such that $(h|_{[0,\frac{1}{2}]})^{k}(U\cap [0,\frac{1}{2}])=[0,\frac{1}{2}]$.
We consider two cases:
\begin{enumerate}
    \item $V\cap [0,\frac{1}{2}]\neq\emptyset$. 
    \smallskip
    
    Then there exists $a\in U$ such that $h^{k}(a)\in V$. Hence,
$F^{k}(a)\cap V\neq\emptyset$.
\item $V\subseteq [\frac{1}{2},1]$. 
\smallskip

Let $z\in U$ be such that 
$h^{k}(z)=\frac{1}{6}$. Since $s$ has a dense orbit on $[\frac{1}{2},1]$ with $h|_{[\frac{1}{2},1]}$, we have that
there exists $m\in\mathbb N$ such that $h^{m}(s)\in V$. Thus, $h^{m}(s)\in F^{k+m}(z)$.
We conclude that $F^{k+m}(z)\cap V\neq\emptyset$.
\end{enumerate} 
Therefore, $F$ is transitive.

\bigskip

Let $M=(\bigcup_{j=0}^{\infty} h^{-j}(\frac{1}{6}))\cup ((\bigcup_{j=0}^{\infty} h^{-j}(\frac{5}{6}))$. If $x\in [0,1]\setminus M$, then $F^{m}(x)=\{h^{m}(x)\}\subseteq [0,\frac{1}{2}]$, if $x\in [0,\frac{1}{2}]$, or 
$F^{m}(x)=\{h^{m}(x)\}\subseteq [\frac{1}{2},1]$, if $x\in [\frac{1}{2},1]$,
for each $m\in\mathbb N$. Hence, $x$ does not have a weak dense orbit.
Let $x\in M$ and let $m\in\mathbb N$ be such that $h^m(x)\in\{\frac{1}{6},\frac{5}{6}\}$. Suppose that $h^m(x)=\frac{1}{6}$. Hence, $x_i\in \{x,h(x),\ldots,\frac{1}{6}\}\cup [\frac{1}{2},1]$, for every $i\in\mathbb N$ where $(x_i)_{i=1}^{\infty}\in \mathcal{O}_F(x)$. Thus, if $U$ is an open subset of $[0,\frac{1}{2}]\setminus\{x,h(x),\ldots,\frac{1}{6}\}$, $F^{i}(x)\cap U=\emptyset$, for all $i\in\mathbb N$. Therefore, $x$ does not have a weak dense orbit. Similarly, we have that $x$ does not have a weak dense orbit when $h^m(x)=\frac{5}{6}$.

\bigskip

Let $X$ be a compactum and let $F\colon X\to 2^{X}$ be an upper semicontinuous function. We say that:
\begin{enumerate}
    \item $F$ is \textit{dense minimal} if each point of $X$ has a dense orbit.
    \item $F$ is \textit{weak dense minimal} if each point of $X$ has a weak dense orbit.
\end{enumerate}

\begin{theorem}\label{theodminimal}
    Let $X$ be a compactum and let $F\colon X\to 2^{X}$ be an upper semicontinuous function. Then 
    $F$ is dense minimal if and only if $F$ is weak dense minimal.
\end{theorem}
\begin{proof}
    By definitions we have that if $F$ is dense minimal, then $F$ is weak dense minimal. We suppose that
    $F$ is weak dense minimal. Let $x\in X$. We prove that $x$ has a dense orbit. Let $\beta=\{B_n : n\in\mathbb N\}$ be a 
    base of $X$. Since $x$ has a weak dense orbit, there exists $k_1\in\mathbb N$ such that $F^{k_1}(x)\cap B_1\neq \emptyset$.
    Let $x_{k_1+1}\in F^{k_1}(x)\cap B_1$. Let $x=x_1,x_2,\ldots,x_{k_1+1}$ be such that $x_{i+1}\in F(x_i)$ for each $i\in\{1,\ldots,k_1\}$.
    Since $x_{k_1+1}$ has a weak dense orbit, there exists $k_2\in\mathbb N$ such that $F^{k_2}(x_{k_1+1})\cap B_2\neq\emptyset$.
    Let $x_{k_1+k_2+1}\in F^{k_2}(x_{k_1+1})\cap B_2$. We define $x_{k_1+2},\ldots,x_{k_1+k_2+1}$ such that $x_{i+1}\in F(x_i)$ for each 
    $i\in \{k_1+1,\ldots, k_1+k_2\}$. Inductively, we define an orbit $(x_n)_{n=1}^{\infty}\in \mathcal{O}_F(x)$ and a sequence of positive integers
    $k_1, k_2, \ldots$ such that $x_{k_1+\cdots +k_{m}+1}\in B_m$ for each $m\in \mathbb N$. Thus, $(x_n)_{n=1}^{\infty}$ is dense.
    Therefore, $F$ is dense minimal.
\end{proof}

\begin{corollary}
    Let $X$ be a compactum and let $F\colon X\to 2^X$ be an upper semicontinuous
function. If $F$ is weak dense minimal, then $F$ is transitive.
\end{corollary}
\begin{proof}
    If $F$ is weak dense minimal, then $F$ is dense minimal, by Theorem \ref{theodminimal}.
    Thus, $F$ is transitive, by Proposition \ref{proptgf6r}.
\end{proof}

\section{Sensitivity}\label{Sensitive}

In this section we use the following notation: if $X$ is a metric space with metric $d$, $A\subseteq X$ and $r>0$, then $\mathcal{V}_r(A)=\{x\in X : d(x,a)<r \text{ for some }a\in A\}$. We write $\mathcal{V}_r(a)$ whenever $A=\{a\}$ for some $a\in X$.

Let $X$ be a continuum with metric $d$. 
A continuous function $f\colon X\to X$ is called \textit{sensitive 
in the sense of Devaney}, provided that there exists $\varepsilon>0$ 
such that for each $x\in X$ and any $\delta>0$ there is $y\in \mathcal{V}_{\delta}(x)$ 
where $d(f^m(x),f^m(y))\geq \varepsilon$ for some $m\in\mathbb N$. 

\medskip

Given an upper semicontinuous function $F\colon X\to 2^X$. We say that:

\begin{enumerate}
    \item $F$ is \textit{strongly sensitive} provided that there exists $\varepsilon>0$ 
    such that for each $x\in X$ and $\delta>0$, there exist $y\in \mathcal{V}_{\delta}(x)$ and $m\in \mathbb N$ such that 
    $$F^m(y)\nsubseteq \mathcal{V}_{\varepsilon}(F^m(x)).$$
    
    \item $F$ is \textit{sensitive} provided that there exists $\varepsilon>0$ 
    such that for each $x\in X$ and $\delta>0$, there exist $y\in\V_\delta(x)$ and $m\in \mathbb N$ such that 
    $$\HH (F^m(x),F^m(y))\geq \varepsilon.$$
    \item $F$ is \textit{weakly sensitive} provided that there exists $\varepsilon>0$ 
    such that for each $x\in X$ and $\delta>0$, there exist $y\in\V_\delta(x)$, $\overline{x}\in\mathcal{O}_F(x)$, $\overline{y}\in\mathcal{O}_F(y)$ and $m\in \mathbb N$ such that 
    $$d (\pi_m(\overline{x}), \pi_m(\overline{y})\geq \varepsilon.$$
\end{enumerate}

Next proposition follows from definitions.

\begin{proposition}\label{prop8uy61}
    Let $F\colon X\to 2^X$ be an upper semicontinuous function. Given the statements:
    \begin{enumerate}
        \item $F$ is strongly sensitive;
        \item $F$ is sensitive;
        \item $F$ is weakly sensitive.
    \end{enumerate}
    Then, \textit{1} implies \textit{2}, and \textit{2} implies \textit{3}.
\end{proposition}

Note that given a continuum $X$ and a constant map $F\colon X\to 2^X$ defined by $F(x)=X$ for each $x\in X$, we have that $F$ is weakly sensitive, but it is not sensitive. Next example shows two functions: a weakly sensitive but not sensitive, and a sensitive for which it is not strongly sensitive.

\begin{lemma}\label{lem1}
    Let $F\colon X\to 2^X$ be an upper semicontinuous function. 
    If there exists a sensitive map, $f\colon X\to X$, in the sense of Devaney such that $f^n(x)\in F^n(x)$ for each $x\in X$ and all $n\in\mathbb N$, then $F$ is weakly sensitive.
\end{lemma}

\begin{proof}
    Since $f$ is sensitive in the sense of Devaney, there exists $\varepsilon>0$ satisfying the definition. 
    Let $x\in X$ and let $\delta>0$. 
    Then we have that there exist $y\in \V_{\delta}(x)$ and $m\in\mathbb N$ such that $d(f^{m}(x),f^{m}(y))\geq\varepsilon$. 
    Let $\overline{x}=\{x,f(x),f^2(x),\ldots\}$ and $\overline{y}=\{y,f(y),f^2(y),\ldots\}$. 
    Note that,  since $f^n(x)\in F^n(x)$, for each $x\in X$ and all $n\in\mathbb N$, $\overline{x}\in \mathcal{O}_F(x)$ and $\overline{y}\in \mathcal{O}_F(y)$. 
    Hence, $d(\pi_{m+1}(\overline{x}),\pi_{m+1}(\overline{y}))=d(f^m(x),f^m(y))\geq \varepsilon$. 
    Therefore, $F$ is weakly sensitive.
\end{proof}

\begin{example}\label{ex00}
    Let $T\colon [0,1]\to [0,1]$ be the tent map defined for each $t\in [0,1]$ by $$T(t)=\begin{cases}
        2t, &\text{ if }t\in [0,\frac{1}{2}];\\
        2(1-t), &\text{ if }t\in [\frac{1}{2},1].
    \end{cases}$$
    It is well known that $T$ is sensitive in the sense of Devaney, and there exists $t_0\in [0,1]$ such that $\{t_0, T(t_0), T^{2}(t_0),\ldots\}$ is a dense subset of $[0,1]$.
    Let $F,G\colon [0,1]\to 2^{[0,1]}$ be given by:
    $$F(t)=\begin{cases}
        [0,1], &\text{ if }t=0;\\
        \{T(t)\}, &\text{ if }t\in (0,1],
    \end{cases} \ \ \text{ and } G(t)=\{t_0,T(t)\}.$$
    
    \begin{affirmation}\label{aff1}
      $F$ is sensitive but it is not strongly sensitive.
    \end{affirmation}
    
    Note that $F^k(0)=[0,1]$ for each $k\in \mathbb N$. Hence, $F^k(y)\subseteq \V_{\varepsilon}(F^k(0))$ for all $k\in\mathbb N$ and each $\varepsilon>0$ and $y\in [0,1]$. Therefore, $F$ is not strongly sensitive. We see that $F$ is sensitive. Since $T$ is sensitive in the sense of Devaney, there exists $\varepsilon>0$ satisfying the definition. We may assume that $\varepsilon<\frac{1}{2}$. Let $x\in [0,1]$ and let $\delta >0$. Observe that if $x=0$, then $\HH (\{T(y), [0,1]\})\geq \frac{1}{2}>\varepsilon$. Hence, we suppose that $x\neq 0$. Since $T$ is sensitive in the sense of Devaney, there exist $y\in \V_{\delta}(x)$ and $m\in\mathbb N$ such that $d(T^m(x),T^m(y))\geq \varepsilon$. We consider three cases:

    \textbf{Case 1.} $F^m(x)=\{T^m(x)\}$ and $F^m(y)=\{T^m(y)\}$. 
    
    In this case, $$\HH (F^m(x),F^m(y))=d(T^m(x),T^m(y))\geq \varepsilon.$$

    \textbf{Case 2.} $F^m(x)=[0,1]$ and $F^m(y)=\{T^m(y)\}$ (or equivalently, $F^m(x)=\{T^m(x)\}$ and $F^m(y)=[0,1]$). 
    
    Then 
    $$\HH (F^m(x),F^m(y))=\HH ([0,1],\{T^m(y)\})\geq\frac{1}{2}>\varepsilon.$$

    \textbf{Case 3.} $F^m(x)=[0,1]$ and $F^m(y)=[0,1]$. 

    In this situation, 
    there exist $r,s< m$ such that $T^r(x)=0$ and $T^s(y)=0$. Since $d(T^m(x),T^m(y))\geq \varepsilon$, we have that $r\neq s$. Suppose that $s<r$ and $r=\min \{i : T^{i}(y)=0\}$. Thus, $F^r(x)=[0,1]$ and $F^r(y)=\{0\}$. Therefore, $\HH (F^r(x),F^r(y))=\HH ([0,1],\{0\})=1\geq\varepsilon$. We proved Affirmation~\ref{aff1}.

    \begin{affirmation}
        $G$ is weakly sensitive but it is not sensitive.
    \end{affirmation}

    Observe that $\{t_0, T(t_0), \ldots, T^{k-1}(t_0)\}\subseteq G^k(x)$, for each $x\in [0,1]$ and every $k\in\mathbb N$. 
    We show that $G$ is not sensitive. Let $\varepsilon>0$. 
    Since $\{t_0, T(t_0), T^{2}(t_0),\ldots\}$ is a dense subset of $[0,1]$, there exists $m\in\mathbb N$ such that $\HH ([0,1],\{t_0, T(t_0),\ldots , T^{m-1}(t_0)\})\leq \frac{\varepsilon}{2}$. 
    Therefore, 
    $$\HH(G^l(x),G^l(y))\leq \HH(G^l(x),[0,1])+\HH([0,1], G^l(y))\leq \frac{\varepsilon}{2}+\frac{\varepsilon}{2}=\varepsilon, \text{ for each }l\geq m.$$ 
    Let $x\in [0,1]$. 
    Since $T, T^2,\ldots,T^{m}$ are maps, there exists $\delta>0$ such that $d(T^{i}(x),T^{i}(y))<\varepsilon$, for each $i\in\{1,\ldots,m\}$ and $y\in \V_{\delta}(x)$. 
    Note that $\HH(G^{i}(x),G^{i}(y))\leq d(T^{i}(x),T^{i}(y))$, for all $i\in\{1,\ldots,m\}$. 
    Therefore, $\HH (G^j(x), G^j(y))<\varepsilon$, for every $y\in \V_{\delta}(x)$ and all $j\in\mathbb N$. 
    Therefore, $G$ is not sensitive.
    Finally, $G$ is weakly sensitive by Lemma~\ref{lem1}.
\end{example}

An upper semicontinuous function $F\colon X\to 2^X$ is {\it Li-Yorke sensitive}
provided there exists an
$\varepsilon>0$ with the property
that any neighbourhood of
any $x\in X$ contains a point $y$ proximal to $x$,
such that trajectories of $x$ and $y$ are separated
by $\varepsilon$ for infinitely many times. Thus, for each $x\in X$ and $\delta>0$ there exists $y\in \V_{\delta}(x)$ such that 
$$\liminf_{n\to\infty}\HH(F^n(x),F^n(y))=0\ \hbox{and}\ 
\limsup_{n\to\infty} \HH(F^n(x),F^n(y))>\varepsilon.$$ 
\begin{proposition}\label{propbshdgst}
    Let $F\colon X\to 2^X$ be an upper semicontinuous function. If $F$ is Li-Yorke sensitive, then $F$ is sensitive. 
\end{proposition}

\begin{proof}
    Let $\varepsilon>0$ be given by the definition of Li-Yorke sensitivity. Let $x\in X$ and let $\delta>0$. Then there exists $y\in\V_{\delta}(x)$ such that  $\limsup_{n\to\infty} \HH(F^n(x),F^n(y))>\varepsilon$. Thus, there exists $m\in \mathbb N$ such that $\HH(F^m(x),F^m(y))\geq \varepsilon$. Therefore, $F$ is sensitive.
\end{proof}

\begin{example}
    Let $F\colon [0,1]\to 2^{[0,1]}$ be the function defined in Example~\ref{ex00}. We know that $F$ is sensitive. We prove that $F$ is not Li-Yorke sensitive.
\medskip

Let $\varepsilon>0$, let $\delta>0$, and let $y\in [0,\delta)$. We consider two cases:
\medskip

\textbf{Case 1.} $F^m(y)=[0,1]$ for some $m\in\mathbb N$. 

Note that $\HH(F^j(0),F^j(y))=0$, for each $j\geq m$. Thus, $\limsup_{n\to\infty} \HH(F^n(x),F^n(y))\leq\varepsilon$.
\vspace{0.5cm}

\textbf{Case 2.} $F^m(y)=\{T^m(y)\}$ for each $m\in\mathbb N$. 

We have that $\HH(F^j(0),F^j(y))\geq \frac{1}{2}$, for all $j\in\mathbb N$. Thus, $\liminf_{n\to\infty}\HH(F^n(x),F^n(y))\neq 0$.
\vspace{0.04cm}

Therefore, $F$ is not Li-Yorke sensitive.
\end{example}



{\bf Acknowledgement.} 
The first named author thanks the support given 
by 
\textit{La Vicerrector\'{\i}a de Investigaci\'on y Extensi\'on 
de la Universidad Industrial 
de Santander y su Programa de Movilidad}. The second named
author thanks the financial support provided by Proyecto 4247 VIE-UIS of \textit{La Vicerrector\'{\i}a de Investigaci\'on y Extensi\'on 
de la Universidad Industrial 
de Santander}.

\noindent (J. Amorocho, J. Camargo)\\   
Escuela de Matem\'aticas, Facultad de Ciencias,\\ 
Universidad Industrial de Santander, Ciudad Universitaria,\\
Carrera 27 Calle 9, Bucaramanga,\\
Santander, A. A. 678, COLOMBIA.\\
e-mail: jeisonamorochom@gmail.com\\
e-mail: jcamargo@saber.uis.edu.co\\

\noindent (S. Mac\'{\i}as)\\
\noindent Instituto de Matem\'aticas,\\ 
Universidad Nacional Aut\'onoma de M\'exico,\\
Circuito Exterior, Ciudad Universitaria,\\
CDMX, C. P. 04510. M\'EXICO.\\ 
e-mail: sergiom@matem.unam.mx \\
e-mail: macias@unam.mx


\begin{thebibliography}{100}




\bibitem{Block} L. S. Block and W. A. Coppel, {\it Dynamics in One Dimension}, Springer Lecture Notes, 1513, Springer Verlag, Berlin, 1992.


\bibitem{CMM} J. Camargo, S. Macías, D. Maya, \textit{Dynamics on Jones’ set function $\mathcal{T}$,}  Topol. Appl. 292, 107635 (2021).

\bibitem{CM} J. Camargo, S. Macías, \textit{Dynamics with Set-Valued Functions and Coselections,} Qual. Theory Dyn. Syst., 21 (25) (2022), pp. 1-43.

\bibitem{CGMO} F. Capulin, Y. Garcia, D. Maya, F. Orozco-Zitli, \textit{Dynamical Properties of upper semicontinuous functions,} Houst. J. Math. 50 (1), 215-236 (2024).











\bibitem{HY} J. G. Hocking and G. S. Young, {\it Topology},  
Dover Publications, Inc.,
New York, 1988.

\bibitem{I} W. T. Ingram, An introduction to inverse limits with set-valued functions, SpringerBriefs Math., Berlin: Springer, 2012.







\bibitem{Kur2} K. Kuratowski, {\it Topology, Vol. II}, Academic Press, New York, 1968.








\bibitem{MT} S. Mac\'{\i}as, {\it Topics on Continua}, 
2nd Edition, Springer-Cham, (2018).

\bibitem{MS} S. Mac\'{\i}as, {\it Set Function $\T$: An Account on F. B. Jones' Contributions to Topology}, Series Developments in Mathematics, Vol. 67, Springer, 2021.










\bibitem{N1} S. B. Nadler, Jr., {\it Hyperspaces of Sets: A Text with Research Questions}, Monographs and Textbooks in Pure and Applied Mathematics, Vol. 49, Marcel Dekker, NewYork, Basel, (1978). Reprinted in: Aportaciones Matem\'aticas de la Sociedad Matem\'atica Mexicana, Serie Textos \#33, (2006).

\bibitem{N2} S. B. Nadler, Jr., {\it Continuum Theory: An
Introduction}, Monographs and Textbooks
in Pure and Applied Mathematics, Vol. 158. New York: Marcel 
Dekker, Inc., 1992.


\bibitem{RT} B. E. Raines, T. Tennant, \textit{The specification property on a set-valued map and its inverse limit,} Houst. J. Math. 44, 665–677 (2018).

\bibitem{SW1} R. M. Schori and J. E. West, $2^I$ is homeomorphic to the Hubert 
cube, Bull. Amer. Math. Soc., 78 (1972), 402-406.










\bibitem{Will} S. Willard, {\it General Topology}, Addison-Wesley Publishing Co., (1970).

\end{thebibliography}
\end{document}